\documentclass{amsart}

\usepackage{mathrsfs}

\usepackage{amsmath,amsthm,amssymb}
\usepackage[margin=1in]{geometry}
\usepackage{graphicx}
\usepackage{caption}
\usepackage{enumitem}
\usepackage{mathrsfs}
\usepackage{xcolor}
\usepackage{multicol}
\usepackage{hyperref}
\usepackage[noabbrev,capitalise,nameinlink]{cleveref}

\newtheorem{thm}{Theorem}[section]
\newtheorem{lem}[thm]{Lemma}
\newtheorem{corollary}[thm]{Corollary}

\newtheorem{example}[thm]{Example}
\newtheorem{prop}[thm]{Proposition}
\newtheorem{proposition}[thm]{Proposition}
\theoremstyle{definition}
\newtheorem{defn}[thm]{Definition}
\newtheorem{rmk}[thm]{Remark}

\newcommand{\rspace}[1]{\mathcal{R}(#1)}
\newcommand{\ospace}{\mathcal{R}_{\mathcal{O}}}

\newcommand{\bR}{\mathbb{R}}

\newcommand{\fJ}{\mathfrak{J}}

\newcommand{\sT}{\mathscr{T}}

\def\ba{\mathbf{a}}
\def\bv{\mathbf{v}}
\def\perm_#1{P(#1)}
\def\wperm_#1{P^W(#1)}
\def\cS{\mathcal{S}}
\DeclareMathOperator{\conv}{conv}
\def\Po{\mathrm{Po}}
\def\gsimp{\mathrm{GenS}}
\def\psimp{\mathrm{PersistS}}

\title{Persistent Subdivisions of Coxeter Permutahedra}

\author[T. Blanton, J. De Loera, M. Sherman-Bennett]{Timothy Blanton, Jes\'us A. De Loera, \and Melissa Sherman-Bennett}

\date{\today}


\usepackage[maxbibnames=99,sorting=nyt,giveninits,maxnames=10,backend=bibtex,block=space]{biblatex}
\usepackage{subcaption}
\begin{document}

\maketitle

\begin{abstract}
    We investigate the realizations of Coxeter permutahedra which are also Coxeter matroid polytopes; these are polytopes of the form $\conv(W \cdot \ba)$ where $W$ is a finite Coxeter group acting on $\mathbb{R}^n$ and $\ba$ is generic. Our main focus is how the geometric properties of $\conv(W \cdot \ba)$ change as $\ba$ changes, with particular attention to persistent simplices, triangulations, and subdivisions.
\end{abstract}

\section{Introduction}

The permutahedron is a classical polytope associated with the symmetric group, first studied by Schoute in 1911 \cite{Schoute1911}. Its vertices are labeled bijectively by the elements of $S_n$ and its geometry reflects the combinatorics of $S_n$. Its one-skeleton is the Hasse diagram of the weak order. Its face lattice is isomorphic to the lattice of cosets of parabolic subgroups (ordered by inclusion), or equivalently, to ordered set partitions of $[n]$ (ordered by coarsening).  The best known realization of the permutahedron is the convex hull of all permutation vectors $\{(w(1), \dots, w(n)): w \in S_n\}$, which we call the \emph{regular permutahedron}. This polytope has a myriad of interpretations and connections to other objects in algebraic and geometric combinatorics (see \cite{BjornerBrenti05,Humphreys90,HolwegLangeThomas2011,Post09,Pons2023} and the references therein). The regular permutahedron is the graphical zonotope associated to the complete graph; the Minkowski sum of the type A roots; the flag matroid polytope of the uniform complete flag matroid; the secondary polytope of a simplicial prism \cite{DLRS10}; and the moment polytope of a generic torus orbit in the complete flag variety.
The normal fan of the regular permutahedron is the famous braid arrangement, providing a connection to algebraic geometry. For instance, \cite{DeConciniProcesi1995} constructs the permutohedral variety---the toric variety whose fan is the braid arrangement---which has played an important role in the Chow theory of matroids.

The permutahedron is a ``type A" object which can be generalized to arbitrary finite Coxeter groups $W$. Such groups act on $\mathbb{R}^n$ by reflections across hyperplanes, and the normal vectors to the hyperplanes form the corresponding (not necessarily crystallographic) root system $\Phi$. The union of the reflecting hyperplanes forms the $W$-braid arrangement (see \cite{BjornerBrenti05,Humphreys90} for more detail). The $W$-permutahedron again has vertices labeled by $W$; one-skeleton equal to the Hasse diagram of the weak order on $W$; and a face lattice isomorphic to the lattice of cosets of parabolic subgroups of $W$. Many of the connections mentioned above between the permutahedron and other topics in algebraic and geometric combinatorics extend to the $W$-permutahedron. For example, there is a \emph{regular $W$-permutahedron}, which is the Minkowski sum of the positive roots of $\Phi$, or alternately, the convex hull of the orbit $W \cdot \ba$ where $\ba$ is the sum of positive roots. The regular $W$-permutahedron has normal fan equal to the $W$-braid arrangement. The regular $W$-permutahedron is a Coxeter matroid polytope for $W$, and when $W$ is crystallographic, it is the moment polytope of a generic torus orbit in $G/B$ \cite{BGW}.

There are many other realizations of the $W$-permutahedron besides the regular one. Recall that the \emph{realization space} $\mathcal{R}(P)$ of a polytope $P$ is a semi-algebraic set inside $(\bR^{\dim P})^{\text{Vertices}(P)}$ encoding (the vertices of) all polytopes with the same face lattice as $P$ (see \cite{Bjorneretal1999, RichterGebert96} for formal definitions). The entire $\mathcal{R}(P)$ is a mysterious, difficult object even for well-studied polytopes such as the $W$-permutahedron. However, since the face lattice of the $W$-permutahedron is preserved under the $W$-action, it is natural to consider the realizations where the vertices themselves are invariant under $W$-action. We call such realizations \emph{orbit $W$-permutahedra}, and the corresponding subset of the realization space the \emph{orbit realization space}; these realizations are the central focus of our results.

It is not hard to show that the $W$-orbit permutahedra are exactly the polytopes $\wperm_{\ba}:=\conv(W \cdot \ba)$, where $\ba$ is a point in $\mathbb{R}^n \setminus H^W$. That is, orbit permutahedra are exactly the convex hulls of generic $W$-orbits. The orbit realization space $\ospace^W$ is the subset of $\rspace{\wperm_\ba}$ consisting of orbit permutahedra. This class of realizations has been studied before. When $W=S_n$, their facet inequalities are described by a classical result of Rado \cite{Rado52}, and their volume and Ehrhart polynomials were studied in the beautiful work of Postnikov~\cite{Post09}. Other combinatorial properties when $W=S_n$ are described in \cite{HolwegLangeThomas2011,Pons2023}. For general $W$, orbit permutahedra are exactly the realizations of the $W$-permutahedron which are Coxeter matroid polytopes in the sense of \cite{BGW}; they are in bijection with points in the interior of the $W$-symmetric $\Phi$-submodular cone \cite{ACEP}.

\vskip .2cm

\paragraph{\bf Our results:} Now we summarize our results.
We first show in (\cref{thm:realization-space}) is that the orbit realization space $\ospace^W$ is homeomorphic to an open ball. This 
opens up the central question of this paper: how do geometric properties of $\wperm_\ba$ change as $\ba$ varies continuously in one chamber of $H^W$? We investigate this question with a particular eye towards properties that \emph{persist} as $\ba$ varies. In particular we investigate which triangulations persist. Triangulations of permutahedra is a topic that has been studied for several years (see, e.g., \cite{BL}).

Our initial motivation for this question was \cite[Theorem 1.5]{DSBW25}, which provided many triangulations of $\wperm_\ba$ that persist for all choices of $\ba$. We characterize the persistent triangulations of $\wperm_\ba$: they are exactly the triangulations whose individual simplices are persistent (\cref{thm:persistent-triangulation}). We then study persistent simplices from a number of perspectives, providing some sufficient conditions and asymptotics. Persistent simplices are closely related to the behavior of the \emph{oriented matroid} \cite{Bjorneretal1999} of the point configuration $W \cdot \ba$ as $\ba$ varies, which turns out to be quite subtle (see Section~\ref{warmup}).

We also present some results on persistent subdivisions and triangulations. As we mentioned above, $\wperm_\ba$ is a Coxeter matroid polytope. We show that all \emph{matroidal} subdivisions of $\wperm_\ba$ are persistent (and in fact one can relax $\wperm_\ba$ to an arbitrary Coxeter matroid polytope $P$ for $W$). When $P$ is a type A flag matroid polytope, these subdivisions are closely related to the tropical flag variety and the flag Dressian \cite{BEZ}; specifically, the \emph{regular} matroidal subdivisions of $P$ correspond to cones in the flag Dressian.

\vskip .2cm

\paragraph{\bf Structure of the paper:} We begin with Section~\ref{sec:why-persistence-matters} on some general reasons why persistence of structures is interesting from a combinatorial perspective. In Section \ref{warmup}, we turn to our specific case of interest, and give definitions and an exploration of $W$-permutahedra for small $n$. This is an experimental investigation of what is preserved or persists under different choices of $\ba$, focusing on simplices and coplanarity. We provide some numerics, including the number of different oriented matroids possible for the point configurations $W \cdot \ba$. Next, in Section \ref{sec:realization} we discuss the contractibility of the orbit realization space, showing that for orbit realizations the space is homeomorphic to an open cone. In Section \ref{persist} we give a sufficient and necessary condition for a triangulation to persist, and some sufficient conditions for simplices to persist. We also briefly discuss the case of ``generic" triangulations and simplices, which persist for most $\ba$ rather than all $\ba$. In Section~\ref{sec:asymptotics} we give some asymptotics on the number of generic and persistent simplices using a ``lifting from facets" argument. In Section~\ref{sec:matroidal-subdivisions} we show that matroidal subdivisions of $W$-matroid polytopes are persistent. 

\vskip .2cm

\paragraph{\bf Acknowledgements:} TB was partially supported by NSF grant DMS-2444020, DMS-2348578, and DMS-2434665. JDL was partially supported by DMS-2348578 and DMS-2434665. MSB was partially supported by DMS-2444020. We would like to thank Paul Breiding for assistance with the Julia package \texttt{HypersurfaceRegions}. An extended abstract advertising some of these results appeared in FPSAC 2026; we thank the FPSAC referees for helpful feedback on the extended abstract.

\section{Interlude: Why Do Persistent Structures matter in Combinatorics?}\label{sec:why-persistence-matters}

The reader may wonder, why should I care about persistent subdivisions, triangulations or normal fans? While the focus of our paper is specifically about persistence for Coxeter realizations of permutahedra and generalized permutahedra, we outline two good reasons why researchers may care about persistent triangulations and subdivisions, or persistent normal fans in general polytopes.

\paragraph{\bf Persistence of triangulations always yields volume and integral formulas} First, the existence of persistent triangulations is in fact very useful to combinatorially compute not just the volume but the value of any integral of a polynomial over a persistent triangulation.

In order to find the formulas or expressions for the volume of any realization of $P_n$, we triangulate it with a persistent triangulation and find the volume of 
each simplex. For example, when the face lattice does not change, a persistent triangulation is the \emph{barycentric triangulation}. This is a classical and widely used triangulation in computational and combinatorial geometry \cite{DLRS10}. Intuitively, after determining the faces of $P$ (this is done through their supporting hyperplanes that define facets), first each two-dimensional face is subdivided by introducing its barycenter and coning to the edges. $k$-dimensional faces are triangulated inductively by coning to the barycenters of all its $k-1$-dimensional bounding faces. This produces the canonical barycentric subdivision of $P$. Once we have a triangulation that is invariant, we can compute the volume or integral of the polytope by computing them in each separate simplex. Note that already the volume formula of a simplex can be directly read from its vertices:

\begin{equation}
    \operatorname{vol}(\Delta)
    =
    \frac{1}{(d-1)!}
    \det \left(\left[ s_2 - s_1, s_3 - s_1, \dots, s_{d} - s_1, \frac{a}{\|a\|}\right]\right),
    \label{eqn:determinant_slice}
\end{equation}
where $a$ is any vector orthogonal to the affine span of the simplex. With two different realizations of $P$ the simplices do not change, thus the above formula does not change and the volume of $P$ is then the sum of the volumes of these persistent simplices.

More generally, to compute the integral of a polynomial $f$, our approach again is to use any  persistent triangulation as outlined before.
For each simplex in the triangulation, we compute formulas for the integral over that simplex using \cite[Lemma 8]{BBDeLKV11}, which provides a closed-form expression for the integral of a power of a linear form over a simplex. Let $M \in Z_{\geq 0}$, and let $\Delta = \conv(s_1, s_2,\dots, s_{n+1})$ be an $n$-dimensional simplex in $\mathbb{R}^d$. For any linear form $\ell$ on $\mathbb{R}^d$,
\begin{equation}
\int_{\Delta} \ell^M \, \mathrm{d}x
  = n!\,\mathrm{vol}(\Delta)\,
    \frac{M!}{(M+n)!}
    \sum_{\substack{k \in \mathbb{N}^{n+1}\\ |k|=M}}
    \langle \ell, s_1 \rangle^{k_1}
    \cdots
    \langle \ell, s_{n+1} \rangle^{k_{n+1}},
\label{eqn:moment}
\end{equation}
where we integrate with respect to the standard Lebesgue measure.
In the case of $M=0$, the integral in the formula above, reduces to computing the volume of a simplex, which we saw only depends on the coordinates of the vertices of the persistent simplex. Finally, to integrate an arbitrary polynomial $f$, we simply need to rewrite it as a sum of powers of linear forms.

\paragraph{\bf When the normal fan is preserved, Ehrhart formulas are preserved.}
In our case, for orbit permutahedra, not only the face lattice is preserved but in fact they have the same \emph{normal fan}, or equivalently, the same \emph{tangent cone}. We recall that for a convex polytope $P \subseteq \mathbb{R}^d$. For each face $F \leq P$, the \textbf{normal cone} of $F$ is
\[
N_P(F) = \{ c \in \mathbb{R}^d : \langle c, x \rangle \geq \langle c, y \rangle \text{ for all } x \in F,\, y \in P \}.
\]
Equivalently, $N_P(F)$ is the set of linear objectives $c$ for which
$F \subseteq \arg\max_{P} \langle c, \cdot \rangle$, i.e., every point of $F$ is optimal.

The \textbf{normal fan} of $P$ is the complete polyhedral fan
\[
\mathcal{N}(P) = \{ N_P(F) : F \leq P \},
\]
whose cones are exactly the normal cones of all faces of $P$ (including $P$ itself, whose normal cone is $\{0\}$, and the empty face, whose normal cone is all of $\mathbb{R}^d$), but the most important cones are the normal cones at vertices because we can recover the others from there. Two polytopes $P$ and $Q$ have the same normal fan if and only if they are \textbf{normally equivalent}, which holds, in particular, when $Q$ is a Minkowski summand of a dilate of $P$. This happens, for instance, in the case of generalized permutahedra. Similarly, for a vertex $v$ of a polytope $P$, the \emph{tangent cone} of $P$ at $v$ is defined as 
\begin{equation*}
    Cone_P(v) := \left\{\, v + w \; | \;  v + \epsilon  w \in P \text{ for some } \epsilon > 0 \, \right\}.
\end{equation*}
The key fact is that if one knows the tangent cone at a vertex one can recover the normal cone at a vertex, and vice versa, using polarity. A motivating example of families of polytopes with the same tangent  normal fans are matroid polytopes \cite{DeLoeraHawsKoeppe} and more generally for generalized W-permutahedra \cite{BjornerBrenti05,Pons2023,CeballosPons2024}. 

For our purposes, having same fan structure within a set of realizations (like Coxeter realizations) is notable because then one can conclude, from the theory developed by Barvinok, Stanley and others \cite{BarvinokPommersheim1999,Barvinok2008,Stanley2012}, that their \emph{Ehrhart polynomials} are, in fact, recoverable in a single parametric formula. Recall that for
$P \subseteq ^n$ be a rational polyhedron. The \emph{multivariate
  generating function} of $P$ is defined as the formal Laurent series in ${\mathbb Z}[[z_1,\ldots,z_n,z_1^{-1},\ldots,z_n^{-1}]]$
\begin{equation*}
\tilde g_P ( \mathbf{z} ) = \sum_{\alpha \in P \cap {\mathbb Z}^n} \mathbf{z}^\alpha ,
\end{equation*}
where we use the multi-exponent notation $ z^{\alpha} =
\prod_{i=1}^n z_i^{\alpha_i}$. This is sometimes called the \emph{integer-point transform} \cite{beckrobins} as its monomials are in bijection with the integer points of $P$. If $P$ is bounded, $\tilde g_P$ is a Laurent polynomial, which we consider as a rational function~$g_P$.
If $P$ is not bounded but is pointed (i.e., $P$ does not contain a
straight line), there is a non-empty open subset $U\subseteq {\mathbb C}^n$ such
that the series converges absolutely and uniformly on every compact
subset of~$U$ to a rational function~$g_\P$ (see \cite{BarvinokPommersheim1999} and
references therein).  If $P$ contains a straight line, we set $g_\Po =
0$.  The rational function $g_P\in {\mathbb Q}(z_1,\dots,z_n)$ defined in this
way is called the \emph{multivariate rational generating function}
of~$P\cap {\mathbb Z}^n$. Barvinok \cite{Barvinok2008} proved that in polynomial 
time, when the dimension of a polytope is fixed, $g_P$ can be represented
as a short sum of rational functions 
\begin{equation*}
 g_P( z) =  \sum_{i \in I} \epsilon_i \frac{z^{ a_{i}}}{\prod_{j=1}^{n} (1 - z^{ b_{ij}} )},
\end{equation*}
where $\epsilon_i \in \{-1,1\}$.
The multivariate generating function formula for the lattice points can be computed from the tangent cones. Moreover, the Ehrhart polynomial can be recovered directly from the rational function expression using special substitutions. See \cite{BarvinokPommersheim1999,Barvinok2008,beckrobins} for details. 

Thus, we can apply the above algorithms to obtain formulas for the volumes or Ehrhart polynomials of permutahedra, matroid polytopes, generalized permutahedra, or any other family of polytopes for which the triangulation, tangent cone or normal fan do not change. 
An illustrative example is the celebrated paper by Postnikov stating formulas the volumes and Ehrhart polynomials of orbit permutahedra \cite{Post09}. Another example: to obtain formulas for Ehrhart polynomials of matroid polytopes, \cite{DeLoeraHawsKoeppe} used the fact that the tangent cones were determined by the combinatorial information alone. 

In summary, if a family of polytopes admits a triangulation whose simplices persist across all realizations, then volumes and any other polynomial integrals can be computed simplex-by-simplex using the same combinatorial subdivision in every realization. Similarly, persistence of the normal fan, or its tangent-cone structure, leads to formulas for Ehrhart and lattice-point formulas for all polytopes with the same fan.

\section{Definitions and warmup} \label{warmup}
 As in the introduction, let $W$ be a finite reflection group acting on a Euclidean space $V$. We follow the conventions of \cite{Humphreys90}.

For each $w\in W$, we may choose a vector $\alpha\in V$ such that $w$ sends $\alpha$ to $-\alpha$. We thus say that $w=s_\alpha$ and $\alpha, - \alpha$ are called \emph{roots} for $W$. The collection of all roots is called a \emph{root system} for $W$, labeled by $\Phi$.

We call a subset of roots \emph{positive} if it consists of exactly the roots on one side of some hyperplane in $V$. A set of positive roots is denoted $\Phi_+$. We call a subset $\Delta$ of roots \emph{simple} if all $\alpha\in \Phi$ are linear combinations of the elements of $\Delta$ with either all coefficients nonpositive or all coefficients nonnegative. Each positive root system contains a unique simple root. Similarly, each simple root system is contained in a unique positive root system. For a fixed choice of positive root system $\Phi_+$ (equivalently, simple system $\Delta$), each positive root system can be written as $w \cdot \Phi_+$ for a unique $w\in W$.

 Note also that each $s_\alpha\in W$ defines a hyperplane $H_\alpha$, consisting of points orthogonal to $\alpha$.

 \begin{defn}
     The hyperplane arrangement $\bigcup_{\alpha \in \Phi}H_\alpha$ is the \emph{$W$-braid arrangement}. The \emph{chambers} of this arrangement are in bijection with positive subsets of roots. For a set of positive roots $\Phi_+$, the corresponding chamber is \[B_{\Phi_+}=\{c\in V\mid \langle c,\alpha\rangle>0 \text{ for all }\alpha\in\Phi_+\}.\] 
     We will fix throughout a chamber $B$ as the fundamental chamber; with this choice, the chambers $B_{\Phi_+}$ are in bijective correspondence with the elements of $W$ and so we also write $B_w$.
 \end{defn}

As in the introduction, we denote $P^W(\ba)=\conv(W\cdot \ba)$ for $\ba\in B$. We first justify that regardless of $\ba \in B$, $P^W(\ba)$ deserves the name of ``$W$-permutahedron".
 
 \begin{lem}\label{lem:consistent-type}
    For any $\ba \in B$, the normal fan of $\wperm_\ba$ is the $W$-braid arrangement. 
\end{lem}

\begin{proof}
    Let $v \in W$ and $\bv = v \cdot \ba$. We start by showing the normal cone of $\bv$ is $\overline{B_v}$.
     Recall that the normal cone of $\bv$, $N_\bv$, is given by \[N_\bv:=\left\{c\in (\bR^d)^\ast\mid  \langle c,\bv\rangle = \underset{y\in \wperm_\ba}{\max} \langle c,y\rangle\right\}\] 
    
    Let $\Delta_v$ be the set of simple roots defining $B_v$. We have that for any $w\in W$, 
    \[\bv-w\cdot\bv=\underset{\alpha\in \Delta}{\sum}k_\alpha \alpha\]
    where $k_\alpha\geq 0$. (see \cite[Page 22]{Humphreys90}). For all $c\in \overline{B_v}$, dotting both sides with $c$ gives us that 
    \[\langle c,\bv-w\cdot\bv\rangle=\left\langle c,\underset{\alpha\in \Delta}{\sum}k_\alpha\alpha\right\rangle=\underset{\alpha\in\Delta}{\sum}k_\alpha \langle c,\alpha\rangle\] 
    Since $k_\alpha\geq 0$ and $\langle c,\alpha\rangle \geq0$, the entire sum is non-negative and so $\langle c,\bv\rangle \geq \langle c,w\cdot\bv\rangle$. As it suffices to compute maximums over the vertices of a polytope, $c\in N_\bv$. Thus $\overline{B_v}\subseteq N_\bv$.

    Let $c\in N_\bv$. By definition, $\langle c,\bv\rangle \geq \langle c,y\rangle $ for all $y\in P$, in particular it holds for $y=w\cdot \bv$. Let $w=s_\alpha$ for some choice of $\alpha\in \Delta$. Then $w\cdot \bv=\bv-2\frac{\langle v,\alpha\rangle}{\langle \alpha,\alpha\rangle}\alpha$. Substituting this into the original inequality gives us that 
    \[\langle c,\bv\rangle \geq \left\langle c,\bv-2\frac{\langle v,\alpha\rangle}{\langle \alpha,\alpha\rangle}\alpha\right\rangle\]
    \[0\geq -2\frac{\langle \bv,\alpha\rangle}{\langle \alpha,\alpha\rangle}\langle c,\alpha\rangle\]
    Since $\bv\in B_v$, $\langle \bv,\alpha\rangle >0$. As $\langle \alpha,\alpha\rangle >0$, we must have that $\langle c,\alpha\rangle \geq0$. Since this holds true for all $\alpha\in \Delta$, we have that $c\in \overline{B_v}$ by definition. Thus $N_\bv=\overline{B_v}$.

    Since the normal cone of each vertex is $\overline{B_v}$, the normal fan of $\wperm_\ba$ is the $W$-braid arrangement and does not depend on the choice of $\ba\in B$. As the normal fan of a polytope determines the combinatorial type, the combinatorial type of $\wperm_\ba$ does not depend on the choice of $\ba\in B$.
\end{proof}

As the vertices of $\wperm_\ba$ are identified with $W$, the vertices of $\wperm_\ba$ and $\wperm_{\ba'}$ are in natural correspondence and it makes sense to talk about structure ``persisting" for different choices of base point. To streamline notation, we introduce the following.

\begin{defn}
For $\cS \subseteq W$ of size $d+1$ and $\ba \in B$, $\cS$ is called an \emph{$\ba$-simplex} if $\conv(\cS \cdot \ba)$ is a $d$-dimensional simplex and is called \emph{$\ba$-coplanar} otherwise. A collection $\{\cS_i\}_{i=1}^p$ of $\ba$-simplices is called an \emph{$\ba$-triangulation} if $\{\conv(\cS_i \cdot \ba )\}_{i=1}^p$ is a triangulation of $\wperm_\ba$. That is, $\wperm_\ba = \cup_{i} \conv(\cS_i \cdot \ba)$ and for all $i,j$,  $\conv(\cS_i \cdot \ba)$ and $\conv(\cS_j \cdot \ba)$ intersect in a common face (possibly empty).   
\end{defn}

We note that the set of $\ba$-simplices is exactly the set of bases of the \emph{matroid} of $\wperm_\ba$. As we shall see in subsequent sections, the set of $\ba$-simplices and of $\ba$-triangulations is heavily dependent on the choice of $\ba$. So we are more interested in the 
following notion.

\begin{defn}
We say that $\cS \subseteq W$ of size $d +1$ is a \emph{persistent simplex} (respectively, is \emph{persistently coplanar}) if it is an $\ba$-simplex for every choice of $\ba \in B$ (respectively, is $\ba$-coplanar for all $\ba \in B$). A collection $\{\cS_i\}$ of persistent simplices is a \emph{persistent triangulation} if $\{\cS_i\}$ is an $\ba$-triangulation for each $\ba \in B$.
\end{defn}

Through several examples we illustrate these concepts in the next three subsections.

\subsection{The 2-dimensional Case}

For $W$ whose associated vector space $V$ is $2$ dimensional, the $W$-permutahedron is a convex $n$-gon. Notice that $n=|W|$. In particular, every $n$-gon whose vertices are invariant under the $W$ action is an orbit $W$-permutahedron. In this case, all 3-element subsets $\cS \subseteq W$ are persistent simplices. In addition, all triangulations are persistent.

The realization space of the $W$-permutahedron is the realization space of a convex $n$-gon, which is contractible. However, not all combinatorial permutahedra are orbit permutahedra. One can easily find hexagons in $\bR^3$ which are not invariant under permuting coordinates. For emphasis, we record this in a proposition.

\begin{prop}
    The realization space of the $W$-permutahedron is strictly larger than the orbit realization space.
\end{prop}

\begin{figure}[h]
    \centering
    \begin{subfigure}[b]{0.4\textwidth}
        \includegraphics[height=0.75\linewidth]{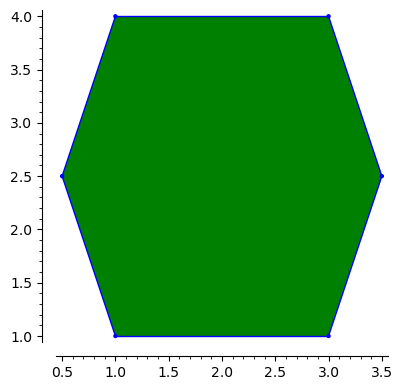}
    \end{subfigure}
    \begin{subfigure}[b]{0.4\textwidth}
        \includegraphics[height=0.75\linewidth]{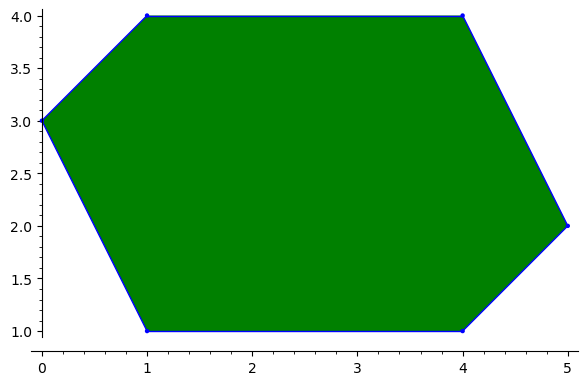}
    \end{subfigure}
    \caption{Hexagons which are not (projections of) an orbit $S_3$-permutahedron.}
    \label{fig:nonorbits}
\end{figure}

\subsection{The $S_4$ case}

Let $W=S_4$. Then we have that $B=\{\ba\in \bR^4\mid a_1<a_2<a_3<a_4\}$. In this case, $\wperm_\ba$ is a $3$-dimensional polytope in $\bR^4$. From the work of Steinitz \cite{Steinitz1922,SteinitzRademacher76}, the realization space of the $S_4$-permutahedron (modulo affine equivalence) is contractible. Again, the orbit realization space is a proper subset of the realization space. The geometry of orbit realizations strongly depends on $\ba$ (see \cref{fig:n=4}), as does the collection of $\ba$-simplices. For example, the permutations $\{(1234), (1342), (2314), (3421)\}$ form an $\ba$-simplex if and only if the polynomial \[a_4^2a_1 - a_4^2a_2 - a_1^2a_2 + a_4a_2^2 + a_1a_2^2 - a_2^3 - 2a_4a_1a_3 + a_1^2a_3 + a_2^2a_3 + a_4a_3^2 - a_2a_3^2\] does not vanish; if the polynomial vanishes, these permutations are instead $\ba$-coplanar. See Example~\ref{ex:dif-number-coplanar} for a different, concrete example of this phenomenon.

\begin{figure}
    \centering
    \begin{subfigure}[b]{0.32\textwidth}
        \centering
        \includegraphics[height=0.75\linewidth]{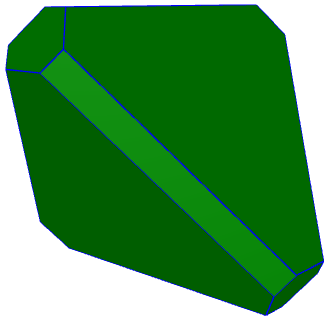}
        \caption{(101714,753589,861628,925275)}
    \end{subfigure}
    \begin{subfigure}[b]{0.3\textwidth}
        \centering
        \includegraphics[height=0.8\linewidth]{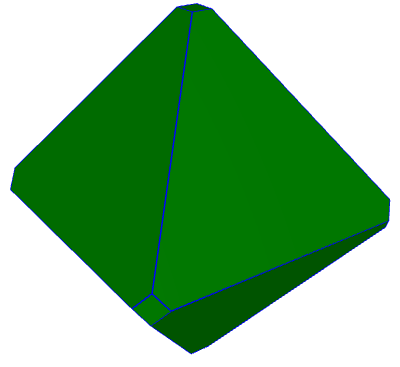}
        \caption{$(32919,67492,433170,464108)$}
    \end{subfigure}
    \hspace{0.3cm}
    \begin{subfigure}[b]{0.25\textwidth}
        \centering
        \includegraphics[height=\linewidth]{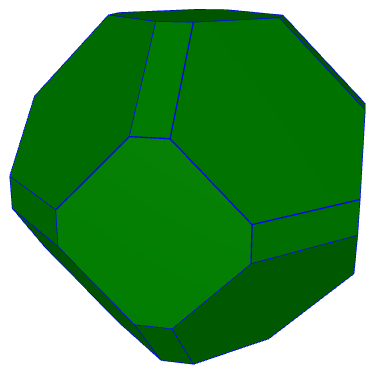}
        \caption{$(1,5,10)$}
    \end{subfigure}
    \caption{(A), (B): $S_4$ orbit permutahedra for different choices of $\ba$. (C) A $B_3$ orbit permutahedron. Displayed beneath each polytope is the choice of base point $\ba$.}
    \label{fig:n=4}
\end{figure}

\begin{figure}[h]
    \centering
    \begin{subfigure}[b]{0.46\textwidth}
        \centering
        \includegraphics[height=0.75\linewidth]{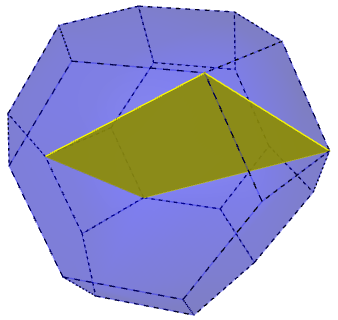}
        \caption{$(1,2,3,4)$}
    \end{subfigure}
    \begin{subfigure}[b]{0.46\textwidth}
        \centering
        \includegraphics[height=0.75\linewidth]{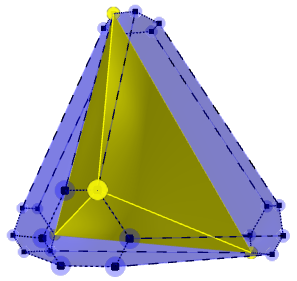}
        \caption{$(101714,753589,861628,925275)$}
    \end{subfigure}
    \caption{The convex hull of $\{(1234),(2143),(3412),(4321)\} \cdot \ba$ in two different orbit permutahedra.}
    \label{fig:Non-dimensional}
\end{figure}

\begin{example}\label{ex:dif-number-coplanar}
    Let $\cS=\{(1234),(2143),(3412),(4321)\}$. Then as shown in \cref{fig:Non-dimensional}, $\cS$ is an $\ba$-simplex for $\ba=(101714,753589,861628,925275)$ but is $\ba'$-coplanar for $\ba'= (1,2,3,4)$. In fact, there are $846$ $\ba'$-coplanar sets, while there are only $474$ $\ba$-coplanar sets.
\end{example}

\begin{prop}\label{prop:s4-classification}
    For subsets of size $4$ in $S_4$, there are $474$ persistently coplanar sets, $9,132$ persistent simplices, and $1,020$ which are simplices for some choices of $\ba$ and coplanar for others. The maximal number of $\ba$-coplanar sets for fixed $\ba$ is $846$, which is achieved for the regular permutahedron.
\end{prop}

To investigate the triangulations of $\perm_\ba$, one might turn to the oriented matroid, or chirotope, of the point configuration $S_4 \cdot \ba$. Assume without loss of generality that the coordinates of $\ba$ do not sum to 0. The \emph{chirotope} of $\perm_\ba$ is the function
\begin{align*}\chi:\{\cS \subseteq S_4:|\cS|=4\} &\to \{0,+, -\} \\
\{v_1, v_2, v_3, v_4\} &\mapsto \mathrm{sign}(\det[v_1 \cdot \ba, v_2 \cdot \ba, v_3 \cdot \ba, v_4 \cdot \ba])
\end{align*}
where we list the elements of $\cS$ according to a fixed (arbitrary) total order of $S_4$.

For $S_4$, there are $278$ distinct chirotopes, with the number of $\ba$-simplices ranging from $9,780$ (for the regular permutahedron) to $10,152$. For generic $\ba$ (e.g. those that avoid the hypersurfaces shown in Figure~\ref{fig:n=4}, left), there are $10,152$ $\ba$-simplices, but they may be oriented differently for $\ba$ in different chambers. The reader can find the computational result in \href{https://github.com/TKing321/Permutahedra}{github:Permutahedra}.

By \cite[Corollary 4.1.44]{DLRS10}, if $\perm_\ba$ and $\perm_{\ba'}$ have the same chirotope, the set of $\ba$-triangulations is equal to the set of $\ba'$-triangulations. We have already seen that the chirotope of $\perm_\ba$ is strongly dependent on $\ba$. Example~\ref{ex:dif-number-coplanar} shows that even the number of sets which $\chi$ sends to $0$ depends on $\ba$. 

\begin{rmk}
    The oriented matroid of an orbit permutahedron is not determined (even up to isomorphism) by its Coxeter group and dimension.
\end{rmk}

The subdivision of $B$ according to the chirotope of $\perm_\ba$ is nonlinear and relatively intricate, as \cref{fig:s4realspace} illustrates.

\begin{figure}[h]
    \centering
    \includegraphics[width=0.48\linewidth]{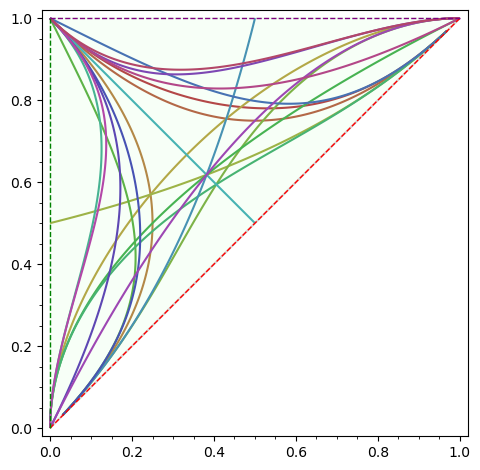}
    \includegraphics[width=0.48\linewidth]{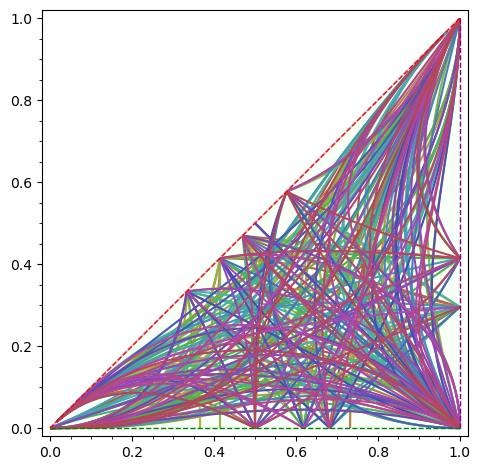}
    \caption{Left: The plane $(0, x, y, 1)$ with the fundamental chamber for $W=S_4$ shaded in green. Right: The plane $(1,x,y)$ with the fundamental chamber of $W=B_3$ in green. Both: The fundamental chamber is subdivided according to the chirotope of $\wperm_\ba$. That is, all $\ba$ in the same region have the same chirotope. When you move to the boundary of a region, some number of simplices collapse to coplanar sets. When you cross one of the curves shown, the chirotope changes sign on some subsets of $S_4$.}
    \label{fig:s4realspace}
\end{figure}

\subsection{The $B_3$ case.}

Let $W=B_3$. We have $B=\{\ba\in R^3\mid a_1>a_2>a_3>0\}$, and $\wperm_\ba$ is a $3$-dimensional polytope in $\bR^3$. Similar to the $S_4$ case, the realization space of the $B_3$-permutahedron is contractible, the orbit realization space is a proper subset of the realization space, and the chirotope of orbit realizations strongly depends on $\ba$.

\begin{prop}\label{prop:b3-classification}
    For subsets of size $4$ in $B_3$, there are $3,600$ persistently coplanar sets, $189,593$ persistent simplices, and $1,387$ subsets which are simplices for some choices of $\ba$ and coplanar for others. 
\end{prop}

For $B_3$, there are $101,119$ distinct chirotopes, with the number of simplices ranging from $187,044$ (for the regular $B_3$ permutahedron) to $190,980$.

\section{Topology of orbit realization space} \label{sec:realization}

Now we would like to give some context for our results.
If $P$ is a 2 or 3-dimensional polytope, $\rspace{P}$ modulo affine transformations is contractible \cite{Steinitz1922} \cite{SteinitzRademacher76}. But in general, $\mathcal{R}(P)$ can be extremely poorly behaved. Indeed, N. Mnëv \cite{Mnev88} proved that realization spaces of convex polytopes (indeed, of oriented matroids) 
can be as complicated as \emph{any} semialgebraic set.  
More precisely, for every primary semialgebraic set $V$ defined over the integers, there exists a (sufficiently high-dimensional) convex polytope $P$ whose realization space $\mathcal{R}(P)$ is “stably equivalent” (in homotopy type, and rationally equivalent) to $V$. In particular, questions about solutions to polynomial equations over $\mathbb{R}$ can be encoded as instances of polytope realization problems. Later work by J. Richter-Gebert \cite{RichterGebertZiegler95,RichterGebert96,RichterGebert99} greatly strengthened the universality phenomenon by showing the complicated topology  already occurs in 4-dimensional polytopes.

However, the situation for the orbit realization space is much simpler, as the following result shows.

\begin{thm}\label{thm:realization-space}
    Let $B$ be a choice of fundamental chamber for some finite reflection group $W$ and let $\ospace^W=\{\conv(W\cdot \ba): \ba\in B\}$ be the orbit realization space. Then $\ospace^W$ is homeomorphic to $B$.
\end{thm}

\begin{proof}

    Fix some arbitrary ordering on $W$ such that the identity is first. We think of $\ospace^W$ as a set of matrices living in $\bR^{n\cdot |W|}$ where each row is a vertex of the polytope. As the vertices are naturally in bijection with elements of $W$, we identify each row with some $w\in W$. As $\ospace^W\subseteq \bR^{n\cdot |W|}$, we endow it with the subspace topology.
    
    Consider the function $f:B\rightarrow\bR^{n\cdot |W|}$ which sends $\ba$ to a matrix where the $w$th row is given by $w\cdot\ba$. This is a linear function in the input variables, so it is continuous. By construction, we have that $\mathrm{im}(f)=\ospace^W$.
    
    For the inverse, note that the map sending a matrix $P\in \bR^{n\cdot |W|}$ to its first row is a projection, so it is continuous. After restricting the domain to $\ospace^W$, the restricted map and $f$ are inverses. As $f$ is a continuous map with continuous inverse, it is a homeomorphism and $\ospace^W\cong B$.
\end{proof}

We note that $\ospace^W$ is clearly in bijection with $B$; the content of the above theorem is that this bijection is topologically well-behaved.

\begin{corollary}\label{cor:contractible}
   The orbit realization space $\ospace^W$ is contractible and thus path-connected.
\end{corollary}

\begin{proof}
    Since $B$ is a open cone, it is contractible and thus path-connected \cite{Hatcher}. As $\ospace^W$ is homeomorphic to $B$, it is also contractible and path-connected.
\end{proof}

\section{Results on persistence of triangulations, simplices, and coplanar sets} \label{persist}

The starting point for our investigation of persistence is the following result. 

\begin{thm}\label{thm:persistent-triangulation}
    An $\ba$-triangulation is a persistent triangulation if and only if it consists of persistent simplices.
\end{thm}

By definition, a persistent triangulation consists of persistent simplices, so one direction of the above theorem is clear. The converse direction is the more interesting one. It tells us that we just have to check that a collection of persistent simplices forms a triangulation for some $\wperm_\ba$, and then we know that they form a triangulation for all orbit permutahedra.  The main tool for this result is the following proposition from \cite{DSBW25}:

\begin{proposition}{\cite[Proposition A.1]{DSBW25}}\label{prop:path-subdivision}
    Let $(P_\tau)_{\tau\in[0,1]}$ be a collection of $d$-dimensional polytopes in $\bR^n$. Let $\fJ_1,\dots,\fJ_m$ be subsets of the vertex set of $P_0$. Suppose that for each $\tau\in[0,1]$, there is bijection $\varphi_\tau$ from the vertex set of $P_0$ to the vertex set of $P_\tau$. Define $Q_\tau^i:=\conv(\varphi_\tau(\fJ_i))$. Assume the following conditions hold:
    \begin{itemize}
        \item $\{Q_0^i\}_{i=1}^m$ is a subdivision of $P_0$
        \item The map $\varphi_0$ is the identity map.
        \item For each vertex $v$ of $P_0$, the map $\tau\mapsto \varphi_\tau(v)$ is a continuous function from $[0,1]$ to $\bR^n$.
        \item For every $\tau\in [0,1]$, $\varphi_\tau$ induces a poset isomorphism from the face lattice of $P_0$ to the face lattice of $P_\tau$.
        \item For every $\tau\in[0,1]$ and $1 \le i \leq m$, $\varphi_\tau$ induces a poset isomorphism from the face lattice of $Q_0^i$ to the face lattice of $Q_\tau^i$.
    \end{itemize}
    Then $\{Q_\tau^i\}_{i=1}^m$ is a subdivision of $P_\tau$ for every $\tau\in[0,1]$.
\end{proposition}

\begin{proof}[Proof of Theorem~\ref{thm:persistent-triangulation}]

Choose $\ba,\ba'\in B$. We will first define a family of polytopes interpolating between $\wperm_{\ba}$ and $\wperm_{\ba'}$. For $\tau \in [0,1]$, let $\ba_\tau :=\tau\ba'+(1-\tau)\ba$. Since $B$ is an open cone, $\ba_\tau\in B$ and we define $P_\tau:=\wperm_{\ba_\tau}$.

Now we define the maps $\varphi_\tau$. For $\bv=v \cdot \ba$ a vertex of $\wperm_\ba$, we let $\bv':=v \cdot \ba'$ be the corresponding vertex of $\wperm_{\ba'}$. Then we define $\varphi_\tau(\bv):=\tau \bv'+(1-\tau)\bv$; in words, $\varphi_\tau(\bv)$ moves along the line segment between $\bv$ and $\bv'$. We have that 
\[\varphi_\tau(\bv)=\tau \bv'+(1-\tau)\bv=\tau (v\cdot\ba')+(1-\tau)(v\cdot\ba)=v\cdot(\tau\ba'+(1-\tau)\ba)=v\cdot\ba_\tau,\] 
so $\varphi_\tau(\bv)$ is a vertex of $P_\tau$. Thus $\varphi_\tau$ is a bijection from $P_0=\wperm_\ba$ to $P_\tau$.

For $\tau=0$, $\varphi_\tau(\bv)=0\bv'+(1-0)\bv=\bv$ which is the identity map. As $\varphi_\tau(\bv):[0,1] \to \bR^n$ is just a parametrization of a line segment, it is continuous.
As $\varphi_\tau$ sends vertices of $P_0$ to vertices of $P_\tau$, it induces a map from the maximal cones of the normal fan $\mathcal{N}(P_0)$ to those of $\mathcal{N}(P_\tau)$, sending the cone $N_{\bv}$ to $N_{\varphi_\tau(\bv)}$. As $\bv=v\cdot \ba$ and $\varphi_\tau(\bv)=v\cdot\ba_\tau$, these maximal cones are the same by \cref{lem:consistent-type}. Since each cone of the normal fan can be written as an intersection of various $N_{\bv}$, $\varphi_\tau$ in fact induces a map from the cones of $\mathcal{N}(P_0)$ to those of $\mathcal{N}(P_\tau)$. This map is also the identity map, and so preserves containment of cones. Since the lattice of cones of $\mathcal{N}(P)$ is dual to the face lattice of $P$, we have that $\varphi_\tau$ induces a poset isomorphism from the face lattice of $P_0$ to the face lattice of $P_\tau$.

Let $\{\cS^i\}$ be an $\ba$-triangulation where $\cS^i$ is a persistent simplex for all $i$. By assumption, $\{Q_0^i=\conv(\cS^i\cdot\ba)\}$ is a triangulation of $P_0=\wperm_\ba$. As $\cS^i$ is a persistent simplex, $Q_\tau^i=\conv(\cS^i\cdot \ba_\tau)$ is a simplex. Since the face lattice of a simplex is the subset lattice, any bijection on the vertices will induce a poset isomorphism.

As we satisfy the conditions of Proposition~\ref{prop:path-subdivision}, $\{Q_\tau^i=\conv(\cS^i\cdot \ba_\tau)\}$ is a subdivision of $P_\tau=\wperm_{\ba_\tau}$ for every $\tau\in[0,1]$. In particular, $\{\cS^i\}$ is an $\ba'$-triangulation for every $\ba'\in B$. Thus $\{\cS^i\}$ is a persistent triangulation.
\end{proof}

Theorem~\ref{thm:persistent-triangulation} shows that the first step to understanding persistent triangulations is to understand persistent simplices. While we do not have a complete characterization, we have the following sufficient conditions. The first is for type $A$.

\begin{prop}\label[prop]{prop:bruhat-graph-persistent-simplices}
    Suppose $\cS \subseteq S_n$ has size $n$. Define 
    \[T:=\{\{i, j\}: \text{ for some }v, w \in \cS,~ vw^{-1}=(i~j))\}.\]
    and consider $T$ as a graph on $[n]$. If $T$ is connected, then $\cS$ is a persistent simplex.
\end{prop}

\begin{proof}
    If $v = s_{\alpha} w$, then the line segment $v \cdot \ba - w \cdot \ba$ is parallel to $e_i-e_j$, regardless of the choice of point $\ba \in B$. Recall that the matroid of the type $A$ positive roots $\{e_i-e_j\}_{i<j}$ is the graphical matroid of the complete graph $K_n$. Since $T$ is a connected subgraph of $K_n$, it contains a spanning tree $\sT$, consisting of $n-1$ edges. Because these edges form a tree, the $n-1$ corresponding root vectors are linearly independent. Since we have $n-1$ linearly independent difference vectors, the $n$ points $\cS \cdot \ba$ are affinely independent and their convex hull is a simplex of dimension $n-1$.
\end{proof}

We can extend this class of persistent simplices to classical types $B/C$ and $D$ by strategically adding a vertex. We use the standard realizations of the classical type root systems, from \cite[2.10]{Humphreys90}. Recall that in this realization, all root systems contain the type $A$ roots $\{e_i - e_j\}_{i \neq j}$.

\begin{prop}\label{prop:bruhat-graph-persistent-simplices-BC}
    Let $W$ be type $B_n$ or $D_n$. Suppose $\cS \subseteq W$ has size $n+1$. Define 
    \[T:=\{\{i, j\}: \text{ for some }v, w \in \cS,~ vw^{-1}=s_{\alpha} \text{ where } \alpha=e_i - e_j\}.\]
    and consider $T$ as a graph on $[n]$. If $T$ is connected and there exists $v,w \in \cS$ such that $v w^{-1} = s_{\beta}$ where $\beta \in \Phi$ is \emph{not} a Type A root, then $\cS$ is a persistent simplex.
\end{prop}

\begin{proof}
    As in the proof of \cref{prop:bruhat-graph-persistent-simplices}, these conditions guarantee that there are $n$ linearly independent distance vectors for the points $\cS\cdot \ba$ for any $\ba$. Of these, $n-1$ vectors will be type $A$ roots (this is guaranteed by the condition on $T$) and one will be the root $\beta$.
\end{proof}

The converse of Proposition~\ref{prop:bruhat-graph-persistent-simplices} does not hold in general---some persistent simplices contain relatively few elements that differ by a reflection. For instance, in $S_4$ there are $3396$ persistent simplices of the type described in Proposition~\ref{prop:bruhat-graph-persistent-simplices}, which is only about $40$ percent of the $8766$ persistent simplices.

We can similarly give a partial description of the persistently coplanar sets in type $A$. For $\cS \subseteq S_n$, let $G_{\cS}$ be the graph with vertices $\cS$ and edge set 
\[\{\{v, w\}: v, w \in \cS \text{ and } vw^{-1} \text{ is a transposition}\}.\] 

\begin{prop}\label{prop:coplanar-sufficient-cond}
    Suppose $\cS \subset S_n$ has size $n$. Suppose there is a forest $F \subseteq G_{\cS}$ such that the graph $T_F$ on $[n]$ with edges 
    \[\{\{i, j\}: vw^{-1}=(i~j)~\text{for some edge }\{v,w\} \in F\}\]
    is not a forest, then $\cS$ is persistently coplanar.
\end{prop}

\begin{proof}
    Since $F$ is a forest on $G_\cS$, if $\cS\cdot \ba$ forms a simplex, then the set of difference vectors $\{v\cdot \ba-w\cdot \ba\mid \{v,w\}\in F\}$ is linearly independent. Again, these line segments are parallel to $e_i-e_j$, regardless of the choice of point $\ba\in B$. By construction, these root directions are all of the edges of $T_F$. If $T_F$ is not a forest, then the roots $\{e_i - e_j: \{i, j\} \in T_F\}$ are linearly dependent, so $\cS$ is a persistently coplanar set.
\end{proof}

One class of triangulations that is known to depend only on the combinatorial type of a polytope is the \emph{lexicographic} triangulations, that is, those that arise from certain ``pulling" and ``placing" operations on each vertex (see \cite[Section 4.3.3]{DLRS10} for more details). A natural question to ask is if these are the only persistent triangulations.

\begin{prop}
    There exist non-lexicographic persistent triangulations.
\end{prop}

\begin{proof} 
    Consider the $\ba$-triangulation of an $S_4$ permutahedron, $P(1,2,3,4)$, into the following simplices:
\begin{multicols}{2}
\begin{enumerate}[label=$S_{\arabic*}$, itemsep=1pt]
    \item $\{(3124), (3142), (3241), (4132)\}$
    \item $\{(3124), (3214), (3241), (4132)\}$
    \item $\{(3124), (3214), (4123), (4132)\}$
    \item $\{(3214), (3241), (4132), (4213)\}$
    \item $\{(3214), (4123), (4132), (4213)\}$
    \item $\{(3241), (4132), (4213), (4231)\}$
    \item $\{(1234), (1243), (1324), (2314)\}$
    \item $\{(1234), (1243), (2134), (3214)\}$
    \item $\{(1234), (1243), (2314), (3214)\}$
    \item $\{(1243), (1324), (1342), (2341)\}$
    \item $\{(1243), (1324), (2314), (2341)\}$
    \item $\{(1243), (2134), (2143), (3214)\}$
    \item $\{(1243), (2143), (3214), (3241)\}$
    \item $\{(1243), (2314), (2341), (3241)\}$
    \item $\{(1243), (2314), (3214), (3241)\}$
    \item $\{(2134), (2143), (3124), (3214)\}$
    \item $\{(2143), (3124), (3142), (3241)\}$
    \item $\{(2143), (3124), (3214), (3241)\}$
    \item $\{(1324), (1342), (1432), (2341)\}$
    \item $\{(1324), (1423), (1432), (2314)\}$
    \item $\{(1324), (1432), (2314), (2341)\}$
    \item $\{(1423), (1432), (2314), (2413)\}$
    \item $\{(1432), (2314), (2341), (2413)\}$
    \item $\{(1432), (2341), (2413), (2431)\}$
    \item $\{(2314), (2341), (2413), (3421)\}$
    \item $\{(2314), (2341), (3241), (3421)\}$
    \item $\{(2314), (2413), (3412), (3421)\}$
    \item $\{(2314), (3214), (3241), (3421)\}$
    \item $\{(2314), (3214), (3412), (3421)\}$
    \item $\{(2341), (2413), (2431), (3421)\}$
    \item $\{(3214), (3241), (3421), (4321)\}$
    \item $\{(3214), (3241), (4213), (4321)\}$
    \item $\{(3214), (3412), (3421), (4312)\}$
    \item $\{(3214), (3421), (4312), (4321)\}$
    \item $\{(3214), (4213), (4312), (4321)\}$
    \item $\{(3241), (4213), (4231), (4321)\}$
\end{enumerate}
\end{multicols}

One can verify using e.g. Sage that this is indeed a triangulation. One can show that for each of these sets $\cS$, the volume of $\conv(\cS \cdot \ba)$ (which is a polynomial in the coordinates of $\ba$) is nonzero on $B$, and so they are persistent simplices. Thus by Theorem~\ref{thm:persistent-triangulation} this is a persistent triangulation. Suppose this is a lexicographic triangulation. For a lexicographic triangulation, there must be some vertex which is last in the ordering and which is either pulled or placed. If the vertex is pulled, then it must be contained in every simplex of the triangulation. A quick check shows that no vertex is contained in every simplex and so the last vertex can not be pulled. Since $P(1,2,3,4)$ is a simple polytope, if the last vertex is placed it will be contained in exactly one simplex. As no vertex is contained in exactly one simplex, the last vertex can not be pulled. This is a contradiction.
\end{proof}

\subsection{Generic permutahedra may have different triangulations}

The above investigation concerns properties which hold for all choices of base point. A natural relaxation of these questions is to instead consider only a dense subsets of base points. Along these lines, we characterize $\ba$-simplices for ``most" $\ba$.

\begin{proposition}\label{prop:generic-simplices}
For generic $\ba \in B$, then the set of $\ba$-simplices is equal to
\[\{\cS \subseteq W: |\cS|=d+1, \cS\text{ not persistently coplanar}\}\]
and in particular does not depend on $\ba$.
\end{proposition}
\begin{proof}
    The volume of $\conv(\cS \cdot \ba)$ is a polynomial in the coordinates of $\ba$. This means that the set of $\ba$ such that $\cS$ is \emph{not} an $\ba$-simplex is Zariski-closed. Thus, if $\cS$ is not persistently coplanar, then $\cS$ is an $\ba$-simplex for $\ba$ in a Euclidean-dense subset of $B$. 
\end{proof}

Proposition~\ref{prop:generic-simplices} may inspire hope that the set of $\ba$-triangulations does not depend on $\ba$ as long as $\ba$ is sufficiently generic; that is, if $\{\cS_i\}$ is an $\ba$-triangulation for one generic $\ba$, it will be an $\ba$-triangulation for a (Euclidean) dense subset of $B$. However, this is far from being true, as we now explain. The key point is that, though the matroid of $\wperm_{\ba}$ is constant for generic $\ba$, the oriented matroid is not, so triangulations do not ``generically persist."

\begin{figure}[htb]
    \centering
    \begin{subfigure}[b]{0.46\textwidth}
        \centering
        \includegraphics[height=0.70\linewidth]{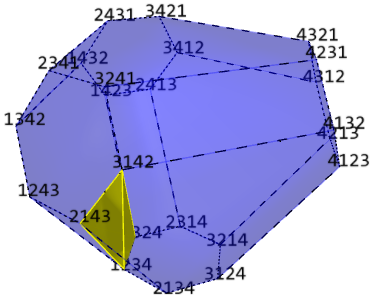}
        \caption{$(1,2,3,6)$}
    \end{subfigure}
    \begin{subfigure}[b]{0.46\textwidth}
        \centering
        \includegraphics[height=0.70\linewidth]{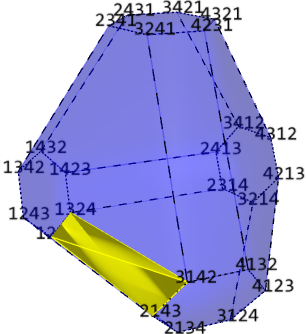}
        \caption{$(1,5,6,7)$}
    \end{subfigure}
    \caption{The convex hull of $\{(1234), (1324), (2143), (3142)\} \cdot \ba$ in two different orbit permutahedra. Note the simplex has negative volume in (a) and positive volume in (b).}
    \label{fig:changesign}
\end{figure}

\begin{prop} There exist $\ba, \ba'$ which are generic in the sense of \cref{prop:generic-simplices}, but the oriented matroids of $W \cdot \ba$ and $W \cdot \ba'$ differ. In particular, there is a collection $\{\cS^i\}$ of subsets of $W$ which is a (regular) $\ba$-triangulation but not a $\ba'$-triangulation, though all $\cS^i$ are $\ba$- and $\ba'$-simplices.
\end{prop}

\begin{proof}

Consider $\cS=\{(1234), (1324), (2143), (3142)\} \subseteq S_4$. Then the signed volume of $\conv(\cS\cdot \ba)$ is positive for $\ba=(1,2,3,6)$ and is negative for $\ba'=(1,5,6,7)$ (see Figure~\ref{fig:changesign}). From this, consider the regular triangulation of $\wperm_\ba$ given by the following simplices.

\begin{multicols}{2}
\begin{enumerate}[label=$S_{\arabic*}$, itemsep=1pt]
    \item $\{(3421), (4231), (4312), (4321)\}$
    \item $\{(3421), (4213), (4231), (4312)\}$
    \item $\{(3412), (3421), (4213), (4312)\}$
    \item $\{(3421), (4132), (4213), (4231)\}$
    \item $\{(3241), (3421), (4132), (4231)\}$
    \item $\{(3421), (4123), (4132), (4213)\}$
    \item $\{(3412), (3421), (4123), (4213)\}$
    \item $\{(3214), (3412), (4123), (4213)\}$
    \item $\{(3241), (3421), (4123), (4132)\}$
    \item $\{(3142), (3241), (4123), (4132)\}$
    \item $\{(3241), (3412), (3421), (4123)\}$
    \item $\{(3142), (3241), (3412), (4123)\}$
    \item $\{(3142), (3214), (3412), (4123)\}$
    \item $\{(3124), (3142), (3214), (4123)\}$
    \item $\{(2431), (3241), (3412), (3421)\}$
    \item $\{(2431), (3142), (3241), (3412)\}$
    \item $\{(2413), (3142), (3214), (3412)\}$
    \item $\{(2413), (2431), (3142), (3412)\}$
    \item $\{(2341), (2431), (3142), (3241)\}$
    \item $\{(2413), (3124), (3142), (3214)\}$
    \item $\{(2314), (2413), (3124), (3214)\}$
    \item $\{(2314), (2413), (3124), (3142)\}$
    \item $\{(2134), (2314), (3124), (3142)\}$
    \item $\{(1432), (2413), (2431), (3142)\}$
    \item $\{(1432), (2341), (2431), (3142)\}$
    \item $\{(1423), (2314), (2413), (3142)\}$
    \item $\{(1423), (1432), (2413), (3142)\}$
    \item $\{(1342), (1432), (2341), (3142)\}$
    \item $\{(1324), (2134), (2314), (3142)\}$
    \item $\{(1324), (1423), (2314), (3142)\}$
    \item $\{(1234), (2134), (2143), (3142)\}$
    \item $\{(1324), (1342), (2143), (3142)\}$
    \item $\{(1234), (1324), (2143), (3142)\}$
    \item $\{(1234), (1324), (2134), (3142)\}$
    \item $\{(1342), (1423), (1432), (3142)\}$
    \item $\{(1324), (1342), (1423), (3142)\}$
    \item $\{(1243), (1324), (1342), (2143)\}$
    \item $\{(1234), (1243), (1324), (2143)\}$
\end{enumerate}
\end{multicols}
    When thinking of these simplices as subsets of $\wperm_{\ba'}$, a straightforward computation yields that $$S_{34}=\{(1,5,6,7),(1,6,5,7),(5,1,6,7),(6,1,7,5)\}\cap S_{38}=\{(1,5,6,7),(1,5,7,6),(1,6,5,7),(5,1,7,6)\}$$ is a $3$ dimensional simplex defined by \[\conv(\{(1,5,6,7),(1,6,5,7),(11/3,8/3,19/3,19/3),(8/3,11/3,19/3,19/3)\}),\] so this is not a triangulation of $\wperm_{\ba'}$ despite all of the simplices being full dimensional.
\end{proof}

\section{Asymptotics on persistent and generic simplices}\label{sec:asymptotics}

Set
\[\gsimp(W):=\#\{\cS\subseteq W\mid \cS \text{ is a }\ba\text{-simplex}\} \quad \text{and} \quad \psimp(W):=\#\{\cS\subseteq W\mid \cS \text{ is a persistent simplex}\}.\]

In this section, we obtain some asymptotics on these quantities.
Clearly,
\[
\psimp(W)
\le
\gsimp(W)
\le
\binom{|W|}{d+1}
\]
where the last binomial coefficient is the number of 
\((d+1)\)-subsets of vertices.
 To obtain asymptotics, our key observation is that persistent simplices in a facet can be lifted to
persistent simplices in the whole polytope.

\begin{lem}\label{lem:facet-lifting}
Let $W_s$ be the maximal parabolic subgroup of $W$ generated by removing a simple reflection $s$. Let $F$ be a facet of $\wperm_\ba$ corresponding to a coset of $W_s$.

Suppose that \(\cS \cdot \ba\subseteq wW_s \cdot \ba\) is a simplex in
this facet for every $\ba\in B$. Then \(\conv(\cS \cdot \ba)\) has \(d\) vertices and spans the facet hyperplane of
\(F\). If \(v\in W\setminus wW_s\), then
\(\cS\cup\{v\}\) is a persistent $\ba$-simplex. 
\end{lem}

\begin{proof}
For every \(\ba\in B\), the vertices indexed by \(\cS\) are affinely
independent and lie in the facet hyperplane of \(F\). Since
\(v\notin wW_s\), the vertex \(v\cdot\ba\) is not contained in this facet hyperplane.
Consequently the \(d\) vertices of \(\cS\) span a hyperplane and the additional
vertex \(v\cdot\ba\) lies outside that hyperplane. Hence the \(d+1\) vertices indexed
by \(\cS\cup\{v\}\) are affinely independent for every \(\ba\in B\).
\end{proof}

\begin{corollary}\label{cor:recursive-lower-bound}
We have that 
\[
\psimp(W)
\ge
\frac{1}{d+1}
\sum_{s\in S}
\frac{|W|}{|W_s|}
\left(|W|-|W_s|\right)
]\psimp(W_{s})
\]
with base case
\[
\psimp(\{e\})=1.
\]
\end{corollary}

\begin{proof}
Fix a simple reflection $s$. The facets of type \(s\) are indexed by the \(|W|/|W_s|\) left
cosets \(wW_s\). Each such facet has exactly \(|W_s|\) vertices, so there are \(|W|-|W_s|\) vertices not in the facet.

Each facet \(F\) is equal to $\conv(wW_s \cdot\ba)$ for some $s$ and some $w$. Since $\conv(wW_s \cdot\ba)$ is a translate of an orbit permutahedron for $W_s$, any persistent simplex for $W_s$ will give a persistent \((d-1)\)-simplex in the facet.
Applying Lemma~\ref{lem:facet-lifting}
to each such simplex and each vertex outside the facet gives
\[
\frac{|W|}{|W_s|}\left(|W|-|W_s|\right)\psimp(W_s)
\]
persistent \(d\)-simplices, counted with multiplicity. Summing over
all simple reflections gives the raw count
\[
R=
\sum_{s\in S}
\frac{|W|}{|W_{s}|}
\left(|W|-|W_{s}|\right)\psimp(W_{s}).
\]

It remains to control multiplicity. Suppose a fixed persistent \(d\)-simplex
\(\cS\subset W\), with \(|\cS|=d+1\), occurs in this raw count. One occurrence is
determined by a choice of the outside vertex \(v\in \cS\), together with a facet
containing the remaining \(d\)-element set \(\cS\setminus\{v\}\). There are only
\(d+1\) choices for \(v\). For a fixed \(v\), the set \(\cS\setminus\{v\}\) is
affinely independent and spans a unique hyperplane. If it lies in a facet and
spans that facet hyperplane, then that facet is uniquely determined: two
facets of a full-dimensional polytope with the same supporting hyperplane are
the same facet. Therefore, each persistent \(d\)-simplex is counted at most
\(d+1\) times in \(R\). Hence the number of distinct persistent
simplices is at least \(R/(d+1)\), which is the asserted bound.

The base case is the rank-zero Coxeter group. Its Coxeter permutahedron is a
single point, and it has exactly one top-dimensional simplex, namely that
point.
\end{proof}

The facet recurrence gives a large, explicit, recursively generated family of
persistent simplices, but it need not account for all persistent simplices.

For an explicit example, take \(W=S_n\) the symmetric group, with orbit permutahedron
\[
P_n
=
\conv\{(a_{\sigma^{-1}(1)},\dots,a_{\sigma^{-1}(n)}):\sigma\in S_n\},
\quad \text{where} \quad
a_1<a_2<\cdots<a_n.
\]
Write 
\[
\psimp_n=\psimp(S_n),
\qquad
\gsimp_n=\gsimp(S_n).
\]
In this case, the upper bound for \((n-1)\)-simplices becomes
\[
\gsimp_n
\le
\binom{n!}{n}.
\]

\begin{lem}\label{lem:asymptotics}
For the type \(A_{n-1}\) Coxeter permutahedron,
\[
((n-1)!)^n \le \psimp_n \le \gsimp_n \le \binom{n!}{n}.
\]
Consequently,
\[
e^{-n+\frac12\log n+O(1)}\binom{n!}{n}
\le \psimp_n
\le
\gsimp_n \le
\binom{n!}{n}.
\]
In particular,
\[
\log \psimp_n
=
 n^2\log n-n^2-\frac12 n\log n+O(n).
\]
\end{lem}

\begin{proof}
The upper bound is the trivial bound by the total number of \(n\)-element
subsets of the vertex set.

For the lower bound, let
\[
G_i=\{\sigma\in S_n:\sigma(i)=1\}.
\]
Thus \(|G_i|=(n-1)!\). For each tuple
\(\cS=(\sigma_1,\dots,\sigma_n)\in G_1\times\cdots\times G_n\), consider the
\(n\)-element set of vertices
\[
\cS\cdot\ba
=
\{v_1,\dots,v_n\},
\qquad
v_i=(a_{\sigma_i^{-1}(1)},\dots,a_{\sigma_i^{-1}(n)}).
\]
The vertex \(v_i\) is the unique vertex in this set whose \(i\)-th coordinate
is \(a_1\). Hence the vertices \(v_1,\dots,v_n\) are distinct, and the tuple
\((\sigma_1,\dots,\sigma_n)\) can be recovered from \(\cS\cdot\ba\). Therefore
this construction gives exactly \(((n-1)!)^n\) distinct \(n\)-subsets of the
vertex set.

It remains to show that each such subset is counted by \(\psimp_n\), i.e.
that its affine-dependence determinant is not identically zero as a polynomial
in the chamber parameters. Fix real numbers
\(b_2<\cdots<b_n\), and set
\[
a^{(T)}=(-T,b_2,\dots,b_n)
\]
for \(T\) sufficiently large, so that \(a^{(T)}\) lies in the open chamber.
For the associated vertices \(v_i(T)\), we have
\[
\frac{1}{T}v_i(T)\longrightarrow -e_i
\qquad\text{as }T\to\infty,
\]
because the \(i\)-th coordinate of \(v_i(T)\) is \(-T\), while all other
coordinates are fixed. The points
\(-e_1,\dots,-e_n\) are affinely independent. Affine independence is an open
condition, and scaling all points by \(1/T\) preserves affine independence.
Therefore, for all sufficiently large \(T\), the vertices
\(v_1(T),\dots,v_n(T)\) are affinely independent. Thus the determinant
polynomial for \(\cS\cdot\ba\) is not identically zero, and
\(\cS\) contributes to \(\psimp_n\). 

Using Stirling's formula and basics of binomials we can justify the statement on asymptotics.

Since \(\psimp_n\) lies between \(((n-1)!)^n\) and \(\binom{n!}{n}\), its
logarithm differs from \(\log\binom{n!}{n}\) by at most \(O(n)\). This gives
\[
\log \psimp_n
=
 n^2\log n-n^2-\frac12n\log n+O(n).
\]
\end{proof}
 \begin{rmk}
     From the above result we have
\[
\psimp_n
=
\exp\left(n^2\log n-n^2-\frac12n\log n+O(n)\right).
\]
Thus, the \(n\)-subsets \(\cS\) for which
\(D_\cS=0\) can change the full binomial count by at most a
factor \(e^{n+O(\log n)}\). On the \(n^2\log n\) scale, the generic count has
the same exponential order as \(\binom{n!}{n}\).
 \end{rmk}

In type $A$, taking just the maximal parabolics of the form $S_{n-1} \times S_1$ in the bound from Corollary~\ref{cor:recursive-lower-bound} gives the bound

\begin{equation} \label{goodeq}
\psimp_n
 \ge \frac{1}{n}\left( \binom{n}{1}
\left(n!-(n-1)!\right)
\psimp(S_{n-1}) + \binom{n}{n-1}
\left(n!-(n-1)!\right)
\psimp(S_{n-1})\right).
\end{equation}

Let \(L_n\) denote the lower bound obtained by iterating the recurrence in Equation \ref{goodeq}.
This quantity satisfies
\[
L_2=1, \text{ and for }n \ge 3,~ L_n
=
2\bigl(n!-(n-1)!\bigr)L_{n-1}=
2(n-1)(n-1)!\,L_{n-1}.
\]
Hence
\[
L_n
=
2^{n-2}\prod_{m=2}^{n-1} m\,m!
=
2^{n-2}(n-1)!\prod_{m=2}^{n-1} m!.
\]

The following table compares this lower bound to $\psimp_n$ and $\gsimp_n$ for small $n$. 

\[
\begin{array}{c|r|r|r|r}
n & L_n & \psimp_n & \gsimp_n & \binom{n!}{n} \\
\hline
3 & 8 & 20 & 20 & 20\\
4 & 288 & 9{,}132 & 10{,}152 & 10{,}626\\
5 & 55{,}296 & ? & 187{,}041{,}024 & 190{,}578{,}024\\
\end{array}
\]

\section{Subdivisions of $W$-matroid polytopes into $W$-matroid polytopes are persistent}\label{sec:matroidal-subdivisions}

In this section, we turn from persistent triangulations to persistent subdivisions. Rather than fixing our attention on orbit permutahedra $\wperm_\ba$, we will consider \emph{orbit polytopes}, which are polytopes of the form $P^\cS(\ba):=\conv(\cS \cdot \ba)$ for $\cS \subset W$ and $\ba \in B$. However, the combinatorial type of an orbit polytope may depend on $\ba$. To remove this complication, we will consider orbit polytopes $P^\cS(\ba)$ that are also generalized permutahedra.

\begin{defn}
    A polytope $P$ is a \emph{generalized $W$-permutahedron} if all of its edges are parallel to roots of $W$, or equivalently, if its normal fan coarsens the $W$-braid arrangement. Following \cite[Section 6.4]{BGW}, an orbit polytope $P^\cS(\ba)$ which is also a generalized $W$-permutahedron is called a \emph{$W$-matroid polytope}\footnote{ \cite{BGW} would use the terminology ``Coxeter matroid polytopes for $(W, 1)$". In type A, if we choose $\ba=(1,2, \dots, n)$, $W$-matroid polytopes $P^{\cS}(\ba)$ are often called (complete) flag matroid polytopes.}.
\end{defn}

For the equivalence of the two definitions of $W$-generalized permutahedra, see e.g. \cite[Proposition 2.6]{ACEP}. Generalized permutahedra and matroid polytopes have been extensively studied, particularly in type A \cite{PRW08,Post09}.
By \cite[Theorem 6.4.1]{BGW}, $P^\cS(\ba)$ is a $W$-matroid polytope if and only if $\cS$ is a \emph{Coxeter matroid} for $(W,1)$. So in particular, whether or not an orbit polytope is a $W$-matroid polytope depends only on $\cS$ and not on $\ba$.
As the next lemma shows, the combinatorial type of a $W$-matroid polytope $P^\cS(\ba)$ also depends only on $\cS$ and not on $\ba \in B$.

For any orbit polytope $P^\cS(\ba)$, the vertices $\cS \cdot \ba$ are in natural bijection with $\cS$. We label the maximal cone $N_{v \cdot \ba}$ of the normal fan with the element $v \in \cS$, and obtain in this way the \emph{labeled normal fan} of $P^\cS(\ba)$.

\begin{lem}\label{lem:consistent-type-gen}
    Fix $\cS\subseteq W$ and $\ba \in B$. If $P^\cS(\ba)$ is a $W$-matroid polytope, its labeled normal fan does not depend on the choice of $\ba\in B$. In particular, its combinatorial type does not depend on $\ba \in B$.
\end{lem}

\begin{proof} 
    Let $v\in W$ and $\bv=v\cdot\ba$. Let $\Phi_v=\{\alpha\in v(\Phi_+) \mid s_\alpha v\in \cS\}$. We claim that 
    \[N_\bv=\{c\in (\bR^n)^\ast\mid \langle c,\alpha\rangle \geq0 \text{ for all }\alpha\in \Phi_v\}.\]

    Take $c$ such that $\langle c,\alpha\rangle\geq 0$ for all $\alpha\in \Phi_\bv$. We have that for any $y\in P^\cS(\ba)$, $y$ is in the vertex span of $\bv$. Thus we have that 
    \[\bv-y=\underset{\alpha\in \Phi_\bv}{\sum}k_\alpha \alpha\]
    where $k_\alpha\geq 0$. Dotting both sides with $c$ gives us that 
    \[\langle c,\bv-y\rangle=\left\langle c,\underset{\alpha\in\Phi_\bv}{\sum}k_\alpha \alpha\right\rangle = \underset{\alpha\in\Phi_\bv}{\sum}k_\alpha \langle c,\alpha\rangle.\]
    Since $\langle c,\alpha\rangle \geq 0$ and $k_\alpha\geq 0$, we have that $\langle c,\bv\rangle\geq \langle c,y\rangle$. Thus $c\in N_\bv$.

    Now let $c\in N_\bv$. We have that $\langle c,\bv\rangle\geq \langle c,y\rangle$ for all $y\in P$. Let $\alpha\in \Phi_\bv$ and consider $s_\alpha\cdot \bv=\bv-2\frac{\langle \bv,\alpha\rangle}{\langle \alpha,\alpha\rangle}\alpha$. Substituting this into the original equation gives us that 
    \[\langle c,\bv\rangle \geq \left\langle c,\bv-2\frac{\langle \bv,\alpha\rangle}{\langle \alpha,\alpha\rangle}\alpha\right\rangle\] \[0\geq -2\frac{\langle \bv,\alpha\rangle}{\langle \alpha,\alpha\rangle}\langle c,\alpha\rangle.\]
    Since $\langle \alpha,\alpha\rangle >0$ and $\langle \bv,\alpha\rangle >0$, we must have that $\langle c,\alpha\rangle \geq 0$. Since this holds true for all $\alpha\in \Phi_\bv$, we get that 
    \[N_\bv=\{c\in (\bR^n)^\ast\mid \langle c,\alpha\rangle \geq0 \text{ for all }\alpha\in \Phi_\bv\}.\]
    
    We have just shown that the normal cone of the vertex $v \cdot \ba$ depends only on $v$ and not on the choice of $\ba\in B$. Thus the labeled normal fan of $P^\cS(\ba)$ does not depend on $\ba$. As the normal fan of a polytope determines the combinatorial type, the second sentence of the lemma follows.
\end{proof}

\begin{rmk}
    One well-studied example of orbit polytopes $P^\cS(\ba)$ which are also $W$-matroid polytopes are \emph{Bruhat interval polytopes}, where $\cS$ is a Bruhat interval in $W$ \cite{TW15}. So in particular the results in this section apply to Bruhat interval polytopes.
\end{rmk}

Suppose $P^\cS(\ba)$ is a $W$-matroid polytope. 
We denote by
\[\ospace^\cS:= \{\conv(\cS\cdot\ba):\ba\in B\}\] the realization space of $W$-matroid polytopes which have the same normal fan as $P^\cS(\ba)$.

\begin{thm}\label{thm:gen_perm-orbit-realization-space}
    Let $B$ be a choice of fundamental chamber for some finite reflection group $W$ and suppose for $\ba \in B$, $P^{\cS}(\ba)$ is a $W$-matroid polytope. Then the realization space $\ospace^\cS$ is homeomorphic to $B$.
\end{thm}

\begin{proof}
    Fix some arbitrary ordering on $\cS$ and let $w'$ be the first element. We think of $\ospace^\cS$ as a set of matrices living in $\bR^{n\cdot |\cS|}$ where each row is a vertex of the polytope. As the vertices are naturally in bijection with elements of $\cS$, we identify each row with some $w\in \cS$. As $\ospace^\cS\subseteq \bR^{n\cdot |\cS|}$, we endow it with the subspace topology.
    
    Consider the function $f:B\rightarrow\bR^{n\cdot |\cS|}$ which sends $\ba$ to a matrix where the $w$th row is given by $w\cdot\ba$. This is a linear function in the input variables, so it is continuous. By construction, we have that $\mathrm{im}(f)=\ospace^\cS$.
    
    For the inverse, note that the map sending a matrix $P\in \bR^{n\cdot |\cS|}$ to its first row is a projection, so it is continuous. Also note that acting on $\bR^n$ by $w'$ is a reflection, so it is continuous. After restricting the domain to $\ospace^\cS$, the composition of these maps and $f$ are inverses. As $f$ is a continuous map with continuous inverse, it is a homeomorphism and $\ospace^\cS\cong B$.
    
\end{proof}

\begin{corollary}\label{cor:contractible-gen-perm}
   The orbit realization space $\ospace^\cS$ is contractible and thus path-connected.
\end{corollary}

These generalize the results of Theorem~\ref{thm:realization-space} and Corollary~\ref{cor:contractible} to the setting of $W$-matroid polytopes. 

We now turn to subdivisions. A collection of $d$-dimensional polytopes $\{Q_i\}_{i=1}^r$ form a \emph{subdivision} of a $d$-dimensional polytope $P$ if $\cup_{i=1}^r Q_i = P$ and for all $i,j$, $Q_i \cap Q_j$ is a face of both $Q_i$ and $Q_j$. We are interested in subdivisions of a $W$-matroid polytope $P^\cS(\ba)$ that ``persist" for different choices of $\ba\in B$.

\begin{defn}
    Fix $\cS \subset W$ and a collection $\Theta=\{\cS_i\}_{i=1}^p$ of subsets of $W$. We say that $\Theta$ is a \emph{$\ba$-subdivision} of $\cS$ if $\{P^{\cS_i}(\ba)\}_{i=1}^p$ is a subdivision of $P^\cS(\ba)$. We say $\Theta$ is a \emph{persistent subdivision} of  $\cS$ if it is a $\ba$-subdivision of $\cS$ for every $\ba\in B$.
\end{defn}

We now turn to $\ba$-subdivisions of a $W$-matroid polytope into $W$-matroid polytopes. Such subdivisions, particularly the \emph{regular} ones, have been studied extensively in type $A$ for a particular choice of $\ba = (1,2,\dots, n)$, and are closely related to the flag Dressian and the tropical complete flag variety. Our main result in this area shows that all such subdivisions are persistent.

\begin{thm}
    Let $\cS, \cS_1, \dots, \cS_r \subset W$ and $\ba \in B$. Suppose $P^{\cS}(\ba)$ and $P^{\cS_i}(\ba)$ are $W$-matroid polytopes. Then $\Theta=\{\cS_i\}_{i=1}^p$ is a $\ba$-subdivision of $\cS$ if and only if it is a persistent subdivision.
\end{thm}

\begin{proof}
    The only statement that needs proof is that an $\ba$-subdivision will be an $\ba'$-subdivision for all $\ba'$. This is essentially identical to the proof of Theorem~\ref{thm:persistent-triangulation}, using Proposition~\ref{prop:path-subdivision} with the same family of maps $\varphi_\tau$. To see that $\varphi_\tau$ induces a poset isomorphism for all of the polytopes, use the argument in the third paragraph of that proof with the reference to Lemma~\ref{lem:consistent-type} replaced by a reference to Lemma~\ref{lem:consistent-type-gen}.
\end{proof}

\printbibliography

\end{document}